\documentclass[reqno]{amsart} 
\usepackage[english]{babel}
\usepackage{graphicx}
\usepackage{amsmath}
\usepackage{amsthm}
\usepackage{amssymb}
\usepackage{stmaryrd}
\usepackage{enumitem}
\usepackage[utf8]{inputenc}
\begin{document}
\title{On some results on Practical Numbers}

\newtheorem{theorem}{Theorem}[section]
\newtheorem{definition}[theorem]{Definition}
\newtheorem{question}[theorem]{Question}
\newtheorem{lemma}[theorem]{Lemma}
\newtheorem{problem}[theorem]{Problem}
\newtheorem{corollary}[theorem]{Corollary}
\newtheorem{example}[theorem]{Example}
\newtheorem{proposition}[theorem]{Proposition}

\newtheorem{NoPolynomialRepresentation}[theorem]{Proposition}
\newtheorem{ArithmeticProgressions}[theorem]{Theorem}

\newcommand{\modd}[1]{\ \normalfont{\text{(mod }}#1)}
\newcommand{\F}[1]{\left\lfloor#1\right\rfloor}
\author{Sai Teja Somu, Ting Hon Stanford Li \and Andrzej Kukla}
\date{\today}
\maketitle
\begin{abstract}
A positive integer $n$ is said to be a practical number if every integer in $[1,n]$ can be represented as the sum of distinct divisors of $n$. In this article, we consider practical numbers of a given polynomial form. We give a necessary and sufficient condition on coefficients $a$ and $b$ for there to be infinitely many practical numbers of the form $an+b$. We also give a necessary and sufficient for a quadratic polynomial to contain infinitely many practical numbers, using which we solve first part of a conjecture mentioned in \cite{Wu}. In the final section, we prove that every number of $8k+1$ form can be expressed as a sum of a practical number and a square, and for every $j\in \{0,\ldots,7\}\setminus \{1\}$ there are infinitely many natural numbers of $8k+j$ form which cannot be written as sum of a square and a practical number. 
\end{abstract}
\section{Introduction}
 A positive integer $n$ is said to be a practical number if every integer in $[1,n]$ can be expressed as the sum of distinct divisors of $n$. Practical numbers were introduced by Srinivasan in \cite{Srinivasan}. In \cite{Stewart}, Stewart showed that $n=p_1^{\alpha_1}p_2^{\alpha_2}\cdots p_k^{\alpha_k}$, where $p_1<p_2<\cdots<p_k$ are primes and integers $\alpha_i\geq 1$ is practical if either $n=1$ or $p_1=2$ and for all $2\leq i \leq k$ we have $p_i\leq \sigma(p_1^{\alpha_1}p_2^{\alpha_2}\cdots p_{i-1}^{\alpha_{i-1}})+1$ (here $\sigma$ denotes the sum of divisors function). This implies that if $m$ is a practical number then for all natural numbers $n\leq \sigma(m)+1$, $mn$ is a practical number.  We will be using this property several times in this article.
 
 There are many properties of practical numbers that are similar to that of properties of prime numbers. All practical numbers except for $1$ are even while all primes except for $2$ are odd. Let $\pi(x)$ and $P(x)$ denote number of primes less than or equal to $x$ and number of practical numbers less than or equal to $x$ respectively. From prime number theorem, we have $\pi(x)\sim \frac{x}{\log x}$ and Weingartner in \cite{Weingartner} showed that there exists a positive constant $c$ such that $P(x)\sim \frac{cx}{\log x}$.
 
 Margenstern in \cite{Margenstern} made a Goldbach-type conjecture that every even positive integer can be expressed as  sum of two practical numbers. In the same paper, he conjectured that there are infinitely many positive integers $m$ such that $m-2,m,m+2$ are all practical numbers. Melfi proved both the conjectures in \cite{Melfi}.
 
 In Section 2, we consider practical numbers of a specific form. We prove that there exists no non-constant polynomial $P(x)$ with integer coefficients such that $P(n)$ is a practical number for all natural numbers $n$. We give a necessary and sufficient condition for any linear polynomial for containing $0$ or $1$ or infinitely many practical numbers. Later we consider quadratic representations of practical numbers. We could prove first part of Conjecture 1.1 of \cite{Wu} using our results.
 
 In Section 3, we prove that every number of the form $8k+1$ can be represented as a sum of a square and a practical number. We also show that for all $j\in \{0,\cdots,7\}$ and $j\neq 1$ there exist infinitely many natural numbers $k$ such that $8k+j$ is not a practical number.

\section{Practical Numbers of a specific form}

Goldbach showed that no non-constant polynomial with integer coefficients can be a prime for all integer values (see \cite{Hardy} and \cite{Nagell}). We show an analogous result for practical numbers.   

\begin{NoPolynomialRepresentation}
For every non-constant polynomial $P$ with integer coefficients and positive leading coefficient there exists a natural number $n$ such that $P(n)$ is not a practical number.
\end{NoPolynomialRepresentation}
\begin{proof}
Let $d$ be the greatest common divisor of $\{P(1)$, $P(2),\cdots\}$. Let prime factorization of $d$ be $p_1^{a_1}\cdots p_k^{a_k}$ with $p_1<p_2<\cdots<p_k$. Let $V=\{p_1,\cdots,p_k\}$ and let $T$ be set of all primes less than or equal to $\sigma(d)+1$ which are not in the set $V$. Since highest power of $p_i$ dividing $d$ is $p_i^{a_i}$, there exists a congruence $x\mod p_i^{a_{i}+1}$ such that $P(x)$ is not a multiple of $p_i^{a_i+1}$. Since $d$ is not divisible by any primes in $T$, for all primes $t\in T$, there exists a congruence $x_t\mod t$ such that $P(x_t)$ is not a multiple of $t$. From Chinese remainder theorem, there exists a congruence $n \mod d \prod_{t\in T} t$ such that $P(n)$ is not divisible by any prime in $T$ and highest power of $p_i$ dividing $P(n)$ is $a_i$. 

Let $m$ be a sufficiently large natural number such that $m\equiv n \mod d \prod_{t\in T} t$  and $P(m)=dq$ where $q>1$ is not divisible by any prime $\leq \sigma(d)+1$. From \cite[Theorem 1]{Stewart}, $P(m)$ is not a practical number.
\end{proof}

\subsection{Practical numbers of the form $an+b$}
Let $a$ and $b$ be two positive integers. In the following theorem we give necessary and condition for $an+b$ to contain infinitely many practical numbers.
\begin{ArithmeticProgressions} Let $a$ and $b$ be two positive integers and
let $d$ be equal to largest practical number divisor of greatest common divisor of $a,b$. We have:
\begin{itemize}
    \item[(a)]If there exists a prime number $p\leq \sigma(d)+1$ such that $p\nmid \frac{a}{d}$ then there exist infinitely many practical numbers of the form $an+b$. 
    \item[(b)]If $b$ is a practical number and all the primes $p\leq \sigma(d)+1$ divide $\frac{a}{d}$ then there is exactly one practical number of the form $an+b$.
    \item[(c)]If $b$ is not a practical number and all the primes $p\leq \sigma(d)+1$ divide $\frac{a}{d}$ then there are no practical numbers of the form $an+b$.
\end{itemize}
\end{ArithmeticProgressions}

\begin{proof}
    Let $a_1=\frac{a}{d}$ and $b_1=\frac{b}{d}$. Let us first prove (a). If there exists a prime $p\leq \sigma(d)+1$ such that $p\nmid a_1$ then we claim that there are infinitely many practical numbers of the form $an+b$. Suppose for a contradiction there are only finitely many practical numbers of the form $an+b$, then there exists an $A$ such that for all $an+b\geq A$, $an+b$ is not a practical number.

   Choose a natural number $k$ such that $p^k\geq A$, $p^k\geq b_1$ and $\sigma(dp^k)+1\geq a_1+1$.
   As $p\nmid a_1$, there exists a solution to $a_1x\equiv -b_1\mod p^k$ and therefore there exists a natural number $1\leq n\leq p^k$ such that $p^k|(a_1n+b_1)$. As $d$ is a practical number and $p\leq \sigma(d)+1$, from \cite[Theorem 1]{Stewart} we have $dp^k$ is a practical number. Now, $an+b=(dp^k)\frac{(a_1n+b_1)}{p^k}$ is a practical number as $dp^k$ is a practical number and \[\frac{a_1n+b_1}{p^k}\leq \frac{a_1p^k+b_1}{p^k} \leq a_1+1\leq \sigma(dp^k)+1.\] As $p^k|an+b$, we have $an+b\geq p^k\geq A$. This contradicts the assumption that there are no practical numbers of the form $an+b$ which are at least $A$. Hence there are infinitely many practical numbers of the form $an+b$.
   This proves (a).
   
   Now let us prove (b) and (c). If for all primes $p\leq \sigma(d)+1$ we have $p|a_1$ then observe that for all primes $p\leq \sigma(d)+1$, we have $p\nmid b_1$ (otherwise $dp$ is a practical number larger than $d$ and is a divisor of $\text{gcd}(a,b)$). Hence $a_1n+b_1$ does not contain any prime factors less than or equal to $\sigma(d)+1$.
   We will now prove that for $n\geq 1$,  $an+b$ is not a practical number. Let $P$ be the smallest prime factor of $a_1n+b_1$. Let the prime factorizations of $d$ and $a_1n+b_1$ be $p_1^{a_1}\cdots p_k^{a_k}$ and $P^{b_1}P_2^{b_2}\cdots P_r^{b_r}$, respectively, where $p_1<p_2<\cdots<P<P_2<\cdots<P_r$. Now, as $P>\sigma(d)+1=\sigma(p_1^{a_1}\cdots p_k^{a_k})+1$, from \cite[Theorem 1]{Stewart} we can conclude that $an+b$ is not a practical number. Therefore, if all primes $p\leq \sigma(d)+1$ are divisors of $\frac{a}{d}$ then $an+b$ is not a practical number for $n\geq 1$. 
   
   Therefore, if $b$ is a practical number and if all primes $p\leq \sigma(d)+1$ are divisors of $a_1$ then there is only one practical number of the form $an+b$, namely $b$. This proves (b). If $b$ is not a practical number and if all primes $p\leq \sigma(d)+1$ are divisors of $a_1$ then there are no practical numbers of the form $an+b$. This proves (c).
\end{proof} 
We have a corollary of Theorem 2.2.
\begin{corollary}
\label{label-5}
Let $a$ be an odd positive integer and $b$ be a positive integer. There exist infinitely many practical numbers of the form of $an+b$.
\end{corollary}
\begin{proof}
Let $d$ be the largest practical number divisor of $\text{gcd}(a,b)$. Since there exists a prime $p=2\leq \sigma(d)+1$ which does not divide $\frac{a}{d}$, from Theorem 2.2 we can conclude that there are infinitely many practical numbers of the form $an+b$.
\end{proof}

\subsection{Practical numbers of the form $an^2+bn+c$} 

Let $q(n)=an^2+bn+c$ be any quadratic polynomial having integer coefficients and a positive leading coefficient. For every prime $p$, let $m_q(p):=\sup\{k\in \mathbb{N}\cup \{0\}:q(n)\equiv 0 \mod p^k \text{ has a solution}\}$ ($m_q(p)=\infty$ is possible). From Hensel's lemma, it follows that there exists a prime number $p$ such that $m_q(p)=\infty$. Let us now work on giving a necessary and sufficient condition for a quadratic polynomial $q(n)$ to represent infinitely many practical numbers. We require a lemma in order to prove Theorem 2.5.  

\begin{lemma}
Let $q(n)=an^2+bn+c$ be any quadratic polynomial with integer coefficients. Let $Q=\{p: \text{ $p$ prime and there exists $n\in \mathbb{N}$ such that $p|q(n)$}\}$ then $\prod_{p\in Q}\big(1+\frac{1}{p}\big)$ diverges.
\end{lemma} 
\begin{proof}
Let $f$ be any irreducible divisor of  $q$ and let $\rho_f(p)$ denote the number of solutions modulo $p$ of the congruence $f(x)\equiv 0 \mod p$. From prime ideal theorem, it follows that $\sum_{p}\frac{\rho_f(p)}{p}$ diverges (see Corollary 3.2.2 of \cite{Murthy}). As $\rho_f(p)\leq 2$, it follows that $\sum_{p\in Q}\frac{1}{p}$ diverges and $\prod_{p\in Q}(1+\frac{1}{p})$ diverges.
\end{proof}
\begin{theorem}
Let $q$ be a quadratic polynomial with integer coefficients and with positive leading coefficient. Let $p_n$ denote $n$th prime number and let $r$ be the least positive integer such that $m_q(p_r)=\infty$. There are infinitely many practical numbers of the form $q(n)$ if and only if $p_1^{m_q(p_1)}p_2^{m_q(p_2)}\cdots p_{r-1}^{m_q(p_{r-1})}p_r$ is a practical number.
\end{theorem} 

\begin{proof}
 Suppose $p_1^{m_q(p_1)}p_2^{m_q(p_2)}\cdots p_{r-1}^{m_q(p_{r-1})}p_r$ is not a practical number. Then either $m_q(p_1)=m_q(2)=0$ or there exists an $2\leq i \leq r$ such that $m_q(p_i)>0$ and $p_i>\sigma\left(p_1^{m_q(p_1)}p_2^{m_q(p_2)}\cdots p_{i-1}^{m_q(p_{i-1})}\right)+1$. If $m_q(2)=0$ then all $q(n)$ are odd and all $q(n)>1$ are not practical numbers. Hence there are only finitely many practical numbers of the form $q(n)$. 
 
 If there exists an $2\leq i \leq r$ such that $m_q(p_i)>0$ and $p_i>\sigma(p_1^{m_q(p_1)}\cdots p_{i-1}^{m_q(p_{i-1})})+1$, then we claim that for all natural numbers $n$ such that $q(n)>p_1^{m_q(p_1)}\cdots p_{r-1}^{m_q(p_{r-1})}p_r$, $q(n)$ is not a practical number. Observe that $q(n)$ should have at least one prime factor $p\geq p_i$ (otherwise from the definition of $m_q(p_i)$, we must have $q(n)\leq p_1^{m_q(p_1)}\cdots p_{i-1}^{m_q(p_{i-1})}$). Let $p$ be the smallest prime factor of $q(n)$ satisfying $p\geq p_{i}$. Now, $q(n)=p_1^{a_1}p_2^{a_2}\cdots p_{i-1}^{a_{i-1}}p^kQ$ where $Q$ is either $1$ or the least prime factor of $Q$ is greater than $p_i$. Now as $p\geq p_i > \sigma(p_1^{a_1}\cdots p_{i-1}^{a_{i-1}})+1\geq \sigma(p_1^{m_q(p_1)}p_2^{m_q(p_2)}\cdots p_{i-1}^{m_q(p_{i-1})})+1$, from \cite[Theorem 1]{Stewart} we can conclude that $q(n)$ is not a practical number. Hence there are only finitely many $q(n)$ such that $q(n)$ is a practical number. This proves one part of the theorem.

If $p_1^{m_q(p_1)}\cdots p_{r-1}^{m_q(p_{r-1})}p_r$ is a practical number. We claim that there are infinitely many practical numbers of the form $q(n)$. For the sake of contradiction, assume that there are only finitely many practical numbers of the form $q(n)$. Then there exists a number $A$ such that there are no practical numbers $q(n)$ such that $q(n)\geq A$.
Let $Q=\{p: \text{ $p$ prime and there exists $n\in \mathbb{N}$ such that $p|q(n)$}\}.$ From previous lemma, it follows that $\prod_{p\in Q}(1+\frac{1}{p})$ diverges. Hence there exist primes $t_1,\ldots, t_s\in Q$ that are greater than $p_r$ such that $\prod_{i=1}^{s}(1+\frac{1}{t_i})>a+b+c$. Let $k$ be a natural number such that $p_r^k>A$ and $p_1^{m_q(p_1)}\cdots p_{r-1}^{m_q(p_{r-1})}p_r^k\geq t_1\cdots t_s$. Since $m_q(p_r)=\infty$ there exists a solution $x_r \mod p_r^k$ such that $q(x_r)\equiv 0 \mod p_r^k$. For $1\leq i\leq r-1$, there exists a solution $x_i\mod p_i^{m_q(p_i)}$ such that $f(x_i)\equiv 0\mod p_i^{m_q(p_i)}$ and for $1\leq j \leq s$ there exists $y_j\mod t_j$ such that $q(y_j)\equiv 0 \mod t_j$. Hence there exists an $1\leq n\leq D$ such that $q(n)$ is divisible by $D$ where $D=p_1^{m_q(p_1)}\cdots p_{r-1}^{m_q(p_{r-1})}p_r^kt_1t_2\cdots t_s$. Observe that as $p_1^{m_q(p_1)}\cdots p_{r-1}^{m_q(p_{r-1})}p_r^k$ is a practical number and since $t_1t_2\cdots t_r\leq p_1^{m_q(p_1)}p_2^{m_q(p_2)}\cdots p_{r-1}^{m_q(p_{r-1})}p_r^k$ we have $D=p_1^{m_q(p_1)}p_2^{m_q(p_2)}\cdots p_{r-1}^{m_q(p_{r-1})}p_r^kt_1t_2\cdots t_s$ is a practical number.

As $q(n)=D(\frac{q(n)}{D})$ and \[\sigma(D)=D\left(\frac{\sigma(D)}{D}\right)\geq n\prod_{j=1}^s\left(1+\frac{1}{t_j}\right)>n(a+b+c)\geq \frac{an^2+bn+c}{n}\geq \frac{q(n)}{n}\geq \frac{q(n)}{D},\] we have $q(n)$ is a practical number greater than $A$, contradicting our assumption. Hence there are infinitely many practical numbers of the form $q(n)$. This proves the other part of the theorem.
\end{proof}
Now before proving a corollary of Theorem 2.5, we prove a lemma.
\begin{lemma}
Let $a,b,c$ be positive integers with $2\nmid ab$ and $2|c$. Then for all positive integers $k$ there exists a natural number $n$ such that $an^2+bn+c\equiv 0 \mod 2^k$.
\end{lemma} 

\begin{proof}
We prove using induction on $k$. For $k=1$, as $a.1^2+b.1+c\equiv 0 \mod 2$, the statement is true for $k=1$. Suppose the statement is true for $k=l$, then there exists an $n$ such that $an^2+bn+c\equiv 0 \mod 2^l$. Now consider $m=2^lx+n$ then as $a(2^lx+n)^2+b(2^lx+n)+c\equiv an^2 + 2^lbx+bn+c\mod 2^{l+1}$, if $x\equiv \frac{an^2+bn+c}{2^l}\mod 2$ then $a(2^lx+n)^2+b(2^lx+n)+c\equiv 0 \mod 2^{l+1}$. Hence the statement is true for $l+1$ and the lemma follows from the principle of mathematical induction.
\end{proof}

\begin{corollary}
Let $a,b,c$ be positive integers with $2\nmid ab$ and $2|c$. Then there are infinitely many practical numbers of the form $an^2+bn+c$.
\end{corollary} 

\begin{proof}
Let $q(n)=an^2+bn+c$ then previous lemma implies that $m_q(2)=\infty$. As $2$ is a practical number, from Theorem 2.5 we can conclude that there are infinitely many practical numbers of the form $an^2+bn+c$.

The above corollary solves the first part of \cite[Conjecture 1.1]{Wu}.
\end{proof}

\subsection{Palindromic practical numbers} 
It is not known whether there are infinitely many palindromic primes but for practical numbers it is easy to show that there are infinitely many palindromic practical numbers. 
\begin{proposition}
There are infinitely many palindromic practical numbers.
\end{proposition} 

\begin{proof}
Let $A_n= \frac{8(10^{2^n}-1)}{9}$. Observe that $A_n$ is palindrome for all natural numbers $n$. We claim that $A_n$ is a practical number for all positive integers $n$. We prove this by using induction on $n$. When $n=1$, $A_1=88$ which is a practical number. Suppose $A_n$ is a practical number, then as $A_{n+1} = A_n(10^{2^n}+1)$ and $\sigma(A_n)\geq 2A_n-1>10^{2^n}+1$ we can conclude that $A_{n+1}$ is also a practical number. Hence $A_n$ is practical for all positive integers $n$. Hence there are infinitely many palindromic practical numbers. 
\end{proof}
	
\section{On Natural Numbers which can be represented as sum of a square and a practical number}
We now prove that every natural number of $8k+1$ can expressed as a sum of practical number and a square. We will prove this using some lemmas.
\begin{lemma}
If $n=2^km$ with $k\geq 1$ and $m\leq 2^{k+1}$, then $n$ is a practical number.
\end{lemma}
\begin{proof}
As $2^k$ is a practical number and $m\leq 2^{k+1}= \sigma(2^k)+1$ we can conclude that $2^km$ is a practical number.
\end{proof}
\begin{lemma}
Let $m$ be a natural number such that $m\equiv 1\mod 8$ and let $k$ be a positive integer. There exists a natural number $1\leq x \leq 2^{k}-1$ satisfying $x^2\equiv m\modd{2^{k+2}}$.
\end{lemma}
\begin{proof}
Let $m$ be any natural number such that $m\equiv 1 \mod 8$. We prove that for all positive integers $k$ there exists a natural number $1\leq x\leq 2^{k}-1$ such that $x^2\equiv m \mod 2^{k+2}$ using mathematical induction. The statement is true for $k=1$ as $1^2\equiv m \mod 2^{3}$. As $m\equiv 1 \mod 8$ implies $m\equiv 1^2\mod 2^4$ or $m\equiv 3^2\mod 2^4$, the statement is true for $k=2$.

Assume that the statement is true for $k=s\geq 2$ then there exists a natural number $1\leq x_s\leq 2^{s}-1$ such that $x_s^2\equiv m \mod 2^{s+2}$. Now, $x_s^2\equiv m \mod 2^{s+3}$ or $x_s^2\equiv m+2^{s+2} \mod 2^{s+3}$.  Let $x_{s+1}=x_s$ if $x_s^2\equiv m \mod 2^{s+3}$ and $x_{s+1}=2^{s+1}-x_s$ if $x_s^2\equiv m+2^{s+2} \mod 2^{s+3}$. Note that $1\leq x_{s+1}\leq 2^{s+1}-1$ and we claim that $x_{s+1}^2\equiv m \mod 2^{s+3}$. If $x_s^2\equiv m \mod 2^{s+3}$ then as $x_{s+1}=x_s$ we have $x_{s+1}^2\equiv m \mod 2^{s+3}$. If  $x_s^2\equiv m +2^{s+2}\mod 2^{s+3}$ then \begin{align*}
x_{s+1}^2 &\equiv (2^{s+1}-x_s)^2\mod 2^{s+3}\\
&\equiv 2^{2s+2}+x_s^2-2^{s+2}x_s\mod 2^{s+3}\\
&\equiv m+2^{s+2}-2^{s+2}x_s\mod 2^{s+3}\\
&\equiv m + 2^{s+2}(1-x_s)\mod 2^{s+3}\\
&\equiv m \mod 2^{s+3} \text{ (as $(1-x_s)$ is even, $2^{s+2}(1-x_s)$ is divisible by $2^{s+3}$)}.
\end{align*}
Hence the statement is true for $k=s+1$ and the lemma follows from the principle of mathematical induction.
\end{proof} 
Now we are ready to prove the main result of this section.
\begin{theorem}
Every natural number of the form $8k+1$ can be expressed as sum of a square and a practical number and for every $j\in \{0,\cdots, 7\}\setminus\{1\}$ there exist infinitely many $k$ such that $8k+j$ cannot be written as sum of a practical number and a square.
\end{theorem}
\begin{proof}
Let $n$ be a natural number $>1$ of the form $8k+1$. Let $m=\F{\log_2\sqrt{8k+1}}$ (this implies $2^{2m}\leq 8k+1<2^{2m+2}$), then as $m\geq 1$, from Lemma 3.2 there exists $1\leq x\leq 2^m-1$ such that $x^2\equiv 8k+1\modd{2^{m+2}}$. Therefore,  $8k+1-x^2=2^{m+2}s$ for some positive integer $s$ (observe that, $s$ is a positive integer as $x^2<2^{2m}\leq 8k+1$). As $2^{m+2}s\leq 8k+1\leq 2^{2m+2}$ we have $s\leq 2^m$. Hence from Lemma 3.1, $2^{m+2}s$ is a practical number and $8k+1= x^2+2^{m+2}s$ is sum of a square and a practical number. This proves the first part of the theorem. 

Let us now prove the second part by  considering each $j\in \{0,\cdots, 7\}\setminus\{1\}$ separately.

\begin{enumerate}
    \setlength\itemsep{0.5em}
    
    \item[($j=0$):]
    Consider numbers $m$ such that $m\equiv 24 \modd {32}$, $m\equiv 2 \modd 3$, $m\equiv 2\modd 5$, $m\equiv -1 \modd 7$, $m\equiv -1\modd {11}$ and $m\equiv 2 \modd {13}$. Then $m$ is of the form $8k$ and $m\equiv x^2 \modd {16}$, $m\equiv x^2 \mod 3$, $m\equiv x^2 \modd 5$, $m\equiv x^2 \modd 7$, $m\equiv x^2 \modd {11}$ and $m\equiv x^2\modd {13}$ have no solutions. Hence if $m=n^2+P$ then $16\nmid P$, $3\nmid P$, $5\nmid P$, $7\nmid P$, $11\nmid P$ and $13\nmid P$. Highest power of $2$ dividing $P$ is less than or equal to $8$, and $P$ is not divisible by any prime less than or equal to $\sigma(8)+1=16$. Hence $P$ cannot be a practical number and $m$ cannot be written as sum of a square and a practical number. Hence there are infinitely many numbers of the form $8k$ which cannot be written as sum of a square and a practical number.
    
    \item[($j=2$):]
    Consider numbers $m$ of the form $24n+2$ such that $24n+1$ is not a perfect square. Then $m\equiv 2 \modd 8$, $m\equiv x^2 \modd 4$ and $m\equiv x^2 \modd 3$ have no solutions. Hence highest power of $2$ dividing $m-x^2$ for any $x$ is less than or equal to $2$ and since $m-x^2$ is not divisible by any prime less than or equal to $\sigma(2)+1=4$, we can conclude that $m-x^2$ is not a practical number. Hence there are infinitely many numbers of the form $8k+2$ which cannot be written as sum of a square and a practical number.
    
    \item[($j=3$):]
    Consider numbers $m$ of the form $24n+11$. Then $m\equiv 3 \modd 8$ and $m\equiv x^2 \modd 4$, $m\equiv x^2 \modd 3$ have no solutions. Hence highest power of $2$ dividing $m-x^2$ for any $x$ is less than or equal to $2$ and since $m-x^2$ is not divisible by any prime less than or equal to $\sigma(2)+1=4$, $m-x^2$ is not a practical number. Hence there are infinitely many numbers of the form $8k+3$ which cannot be written as sum of a square and a practical number.
    
    \item[($j=4$):]
    Consider numbers $m$ such that $m\equiv 12 \modd {16}$, $m\equiv 2 \modd 3$, $m\equiv 2\modd 5$, $m\equiv -1 \modd 7$, $m\equiv -1\modd {11}$, and $m\equiv 2 \modd {13}$. Then $m$ is of the form $8k+4$ and $m\equiv x^2 \modd {16}$, $m\equiv x^2 \modd 3$, $m\equiv x^2 \modd 5$, $m\equiv x^2 \modd 7$, $m\equiv x^2 \modd {11}$, and $m\equiv x^2\modd {13}$ have no solutions. Hence if $m=n^2+P$ then $16\nmid P$, $3\nmid P$, $5\nmid P$, $7\nmid P$, $11\nmid P$ and $13\nmid P$. Highest power of $2$ dividing $P$ is less than or equal to $8$, and $P$ is not divisible by any prime less than or equal to $\sigma(8)+1=16$. Hence $P$ cannot be a practical number and $m$ cannot be written as sum of a square and a practical number. Hence there are infinitely many numbers of the form $8k+4$ which cannot be written as sum of a square and a practical number.
    
    \item[($j=5$):]
    Consider numbers $m$ such that $m\equiv 5 \modd 8$, $m\equiv 2 \modd 3$, $m\equiv 2\modd 5$, and $m\equiv -1 \modd 7$. Then $m$ is of the form $8k+5$ and $m\equiv x^2 \modd 8$, $m\equiv x^2 \modd 3$, $m\equiv x^2 \modd 5$, and $m\equiv x^2 \modd 7$ have no solutions. Hence if $m=n^2+P$ then $8\nmid P$, $3\nmid P$, $5\nmid P$ and $7\nmid P$. Highest power of $2$ dividing $P$ is less than or equal to $4$, and $P$ is not divisible by any prime less than or equal to $\sigma(4)+1=8$. Hence $P$ cannot be a practical number and $m$ cannot be written as sum of a square and a practical number. Hence there are infinitely many numbers of the form $8k+5$ which cannot be written as sum of a square and a practical number.
    
    \item[($j=6$):]
    Consider numbers $m$ of the form $24n+14$. Then $m\equiv 6 \modd 8$, $m\equiv x^2 \modd 4$ and $m\equiv x^2 \modd 3$ have no solutions. Hence highest power of $2$ dividing $m-x^2$ for any $x$ is less than or equal to $2$ and since $m-x^2$ is not divisible by any prime less than or equal to $\sigma(2)+1=4$, we have $m-x^2$ is not a practical number. Hence there are infinitely many numbers of the form $8k+6$ which cannot be written as sum of a square and a practical number.
    
    \item[($j=7$):]
    Consider numbers $m$ of the form $24n+23$. Then $m\equiv 7 \modd 8$,  and $m\equiv x^2 \modd 4$ and $m\equiv x^2 \modd 3$ have no solutions. Hence highest power of $2$ dividing $m-x^2$ for any $x$ is less than or equal to $2$ and since $m-x^2$ is not divisible by any prime less than or equal to $\sigma(2)+1=4$, $m-x^2$ is not a practical number. Hence there are infinitely many numbers of the form $8k+7$ which cannot be written as sum of a square and a practical number.
\end{enumerate}
\end{proof}
\section{Future prospects}
In this paper, we have given necessary and sufficient conditions for linear and quadratic polynomials to represent infinitely many practical numbers (Theorem 2.2 and Theorem 2.5). It would be worthwhile to explore such necessary and sufficient conditions for cubic and  biquadratic polynomials with integer coefficients.
\section{Disclosure Statement}
No potential conflict of interest was reported by the authors.

\section{Data availability statement}
There is no data associated with this article.


\begin{thebibliography}{9}

\bibitem{Murthy} A.C. Cojocaru, M.R. Murty,
{\it An Introduction to Sieve Methods and Their Applications}, 
vol. {\bf 66}, Cambridge University Press (2005).

\bibitem{Hardy} G. H.Hardy, E. M.Wright, {\it An Introduction to the Theory of Numbers}, 5th ed. Oxford, England: Clarendon Press, 1979.

\bibitem{Margenstern} M. Margenstern, {\it Les nombres pratiques: th´eorie, observations et conjectures}, J. Num-
ber Theory {\bf 37} (1991) 1-36.

\bibitem{Melfi} G. Melfi, {\it On two conjectures about practical numbers}, J. Number Theory 56 (1996)
205–210.

\bibitem{Nagell} T.Nagell, {\it Primes in Special Arithmetical Progressions}, §44 Introduction to Number Theory. New York: Wiley, 1951.

\bibitem{Srinivasan} A. K. Srinivasan, {\it Practical numbers}, Current Sci. 17 (1948), 179–180.

\bibitem{Stewart} B. M. Stewart, {\it Sums of distinct divisors}, Amer. J. Math. 76 (1954) 779–785.


\bibitem{Weingartner} A. Weingartner, {\it Practical numbers and the distribution of divisors}, Q. J. Math. {\bf 66}
(2015), 743–758.

\bibitem{Wu} Wu, XH. {\it Special forms and the distribution of practical numbers}, Acta Math. Hungar. 160, 405–411 (2020).
\end{thebibliography}
\end{document}